\documentclass[a4paper,10pt]{article}

\usepackage[utf8]{inputenc}
\usepackage{amsmath,amssymb}
\usepackage{graphicx}
\usepackage{tikz}
\usepackage{subfigure}
\usepackage{algorithm,algorithmic}
\usepackage[textwidth=16cm,textheight=21cm]{geometry}
\usepackage{enumerate}
\usepackage{color}

\newtheorem{theorem}{Theorem}
\newtheorem{lemma}[theorem]{Lemma}

\newtheorem{definition}[theorem]{Definition}

\newtheorem{conjecture}[theorem]{Conjecture}

\newcommand{\qed}{\hfill \ensuremath{\Box}}

\newenvironment{proof}[1][Proof]{\begin{trivlist}
\item[\hskip \labelsep {\bfseries #1}]}{\qed\end{trivlist}}

\newenvironment{proofof}[1][Proofof]{\begin{trivlist}
\item[\hskip \labelsep {\bfseries #1}]}{\qed\end{trivlist}}

\newcommand{\M}{\mathcal{M}}
\renewcommand{\S}{\mathcal{S}}
\newcommand{\E}{{\mathbb E}}

\newcommand{\RRM}{\mathsf{RRM}}
\newcommand{\URM}{\mathsf{URM}}
\newcommand{\wl}{\emph{wlog}}
\newcommand{\whp}{{\it whp}}

\definecolor{verd}{rgb}{0.1,0.5,0.1}

\title{Rainbow Matchings: existence and counting}

\author{Guillem Perarnau and Oriol Serra}

\begin{document}



\maketitle

\begin{abstract} A perfect matching $M$ in an edge--colored complete bipartite graph $K_{n,n}$
 is rainbow if no pair of edges in $M$ have the same color. We obtain  asymptotic enumeration
 results for the number of rainbow matchings in terms of the maximum number of occurrences of a
color. We also consider two natural models of random edge--colored $K_{n,n}$ and show that, if the number of colors is at least $n$, then there is $\whp$ a random matching. This in particular shows that
almost every square matrix of order $n$ in which every entry appears at most $n$ times has a Latin transversal.
\vspace{0.5cm}

\end{abstract}

\section{Introduction}\label{intro}

A subgraph $H$ of an edge--colored graph $G$ is {\it rainbow}
if no color appear twice in $E(H)$. The study of rainbow subgraphs has
a large literature; see e.g.~\cite{ajmp2003,fk2008,KL2008,lsww2010,jw2007}. In this paper we deal
with {\it rainbow perfect matchings} of edge--colored complete bipartite graphs $K_{n,n}$.
These are  equivalent to  latin transversals in square matrices of order $n$, sets of $n$
pairwise distinct entries no two in the same row nor the same column. The following is a
longstanding conjecture by Ryser~\cite{r1967} on the existence of latin transversals in
latin squares:

\begin{conjecture}[Ryser]
 Every latin square of odd order admits a latin transversal.
\end{conjecture}

For even size, there are some latin squares that have no rainbow matchings, such as the
additive table of $\mathbb{Z}_{2n}$. Nevertheless, it was also conjectured (see e.g.~\cite{s1975}) that
every latin square of even size has a partial latin transversal of length $n-1$.

There are different approaches to address these conjectures.
For instance, Hatami and Shor~\cite{hs2008} proved that every latin square has a partial transversal
of size $n-O(\log{n}^2)$.
Another approach was given by Erd{\H{o}}s and Spencer~\cite{es1991} where they prove the
following result:

\begin{theorem}[Erd\H os, Spencer~\cite{es1991}]
	\label{thm:lopsi}
	Let $A$ be square matrix of order $n$. If  every entry in $A$ appears at most
$\tfrac{n-1}{4e}$ times, then $A$ has a latin transversal.
\end{theorem}

In order to get this result the authors  developed the Lopsided version of the Lov\'asz Local Lemma. The main idea of this version is to set a different dependency graph called
{\it lopsidependency graph}. In this graph edges may no longer represent dependencies and the hypothesis of the Local Lemma are replaced by a weaker assumption.

In this paper we address two  problems related to rainbow matchings in edge--colored $K_{n,n}$:
asymptotic enumeration and existence in random edge--colorings. Our edge--colorings are non necessarily proper and the results apply to  proper edge--colorings as well.

\vspace{0.3cm}

Previous results on enumeration of latin transversals have been obtained by McKay, McLeod and Wanless~\cite{mmw2006}, where the authors give 
upper and lower bounds on the maximum number of transversals that a latin square can have. However there is still a large gap between these bounds.
Here, we provide, under some hypothesis,  upper and lower bounds for the probability that a random
matching in an edge--colored $K_{n,n}$ is rainbow, that are asymptotically tight.

The bounds are obtained by techniques inspired by the framework devised by Lu and
Sz\'ekely~\cite{ls2009} to obtain asymptotic enumeration results
with the Local L\'ovasz Lemma.

For an edge--coloring of the complete bipartite graph $K_{n,n}$, we let $\M$ denote  the family of  pairs of non incident edges that have the same color.  Let $M$ be a random matching of  $K_{n,n}$. For each $(e,f)\in\M$, we denote by $A_{ef}$ the event  that the pair of edges  $e,f$ belongs to $M$.
This is the set of bad events in the sense that, if none of these events occur, then the matching $M$ is
rainbow.

Therefore, given a set of bad events $A_1,\dots,A_m$, we consider the problem of estimating the
probability of the event $\cap_{i=1}^m \overline{A_i}$. If  the bad events are mutually independent, then the number of bad events that are satisfied follows a Poisson distribution with parameter $\mu=\sum_{i=1}^m \Pr(A_i)$. Hence,
\begin{equation*}
\Pr( \cap_{i=1}^m \overline{A_j}) = e^{-\mu}.
\end{equation*}

It is natural to expect a similar behaviour if the dependencies among the events are rare. This is
known as the Poisson Paradigm (see e.g.~\cite{as2008}). Our objective is to show that
\begin{equation*}
\Pr( \cap_{i=1}^m \overline{A_j})  \rightarrow e^{-\mu} \quad(n\rightarrow \infty).
\end{equation*}
Let $X_M$ denote the indicator variable that a random perfect
matching $M$ is rainbow in a fixed edge--coloring of $K_{n,n}$. Let ${\mathcal M}$ denote the set of
pairs of independent edges that have the same color (bad events.) Our first result is the following:

\begin{theorem}
\label{thm:given}
Fix an edge--coloring  of $K_{n,n}$ such that no color appears more than $n/k$ times, where $k=k(n)$. Let $\mu=|\M|/n(n-1)$.

If $k\ge 12$ then there exist  constants $c_1<1<c_2$ depending only in $k$, such
that
$$
e^{-c_2\mu}\leq \Pr(X_M=1) \leq e^{-c_1\mu}.
$$
In particular, if $k=\omega(1)$ then
$$
\Pr(X_M=1) = e^{-(1+o(1))\mu}.
$$
Moreover, if $k=\omega(n^{1/2})$ then
$$
\Pr(X_M=1)=e^{-\mu}(1+o(1)).
$$
\end{theorem}

In the proof of Theorem \ref{thm:given} we obtain  $c_1=1-2/k-12/k^2$ and $c_2=1+16/k$.
Note that the probability of having a rainbow matching only depends on the number of bad events that
the given coloring defines. Perhaps surprisingly, this probability does not depend on the structure of
the set of bad events in the coloring.

\vspace{0.3cm}

The results in Theorem \ref{thm:given} require the condition $k\ge 12$, which is one unit more than the  one given by Erd\H os and Spencer \cite{es1991} for the existence of rainbow matchings. This prompts us to analyze the existence of rainbow matchings in {\it random} edge--colorings of $K_{n,n}$ in the more general setting when $k\ge 1$ (we can not use less than $n$ colors.) For the existence of rainbow matchings in random edge--colorings of $K_{n,n}$ we restrict ourselves to colorings with a fixed number $s=kn$ of colors. We define two natural random models that fit with this condition.

In the Uniform random model, $\URM$, each edge gets one of the $s$ colors independently and
uniformly at random. In this model, every possible edge coloring with at most $s$ colors  appears
with the  the same probability. In the Regular random model, $\RRM$, we choose an edge coloring
uniformly at random among all the equitable edge colorings, where each color class has prescribed
size $\tfrac{n}{k}$. Although they have the same expected behaviour, both models are interesting.
A result analogous to the one in Theorem~\ref{thm:given} can be proven for these two models.

\begin{theorem}
\label{thm:random}
 Let $c$ be a random edge coloring of $K_{n,n}$ in the model $\URM$ with $s\ge n$ colors. Then,

$$
\Pr(X_M=1) = e^{-c(k)\mu}
$$
where $\mu\sim \tfrac{n^2}{2s}$ and $$c(k)=2k\left( 1 - (k-1)\log
\left(\frac{k}{k-1}\right)\right)$$

Let $c$ be a random edge-coloring of $K_{n,n}$ in the $\RRM$ model with $s\ge n$ colors. Then
$$
\Pr(X_M=1) = e^{-(c(k)+o(1))\mu}
$$
\end{theorem}

Observe that both models lead to similar results. In particular, if $k=\omega(1)$ 
$$
\Pr(X_M=1) = e^{-(1+o(1))\mu}
$$

The $\RRM$ behaves as expected since, as we have observed, just the number of bad events is
relevant, and in this case it is approximately $\tfrac{n^3}{2k}$.

Since the colorings are random, we have a stronger concentration of the rainbow matching
probability than in the case of fixed colorings. By using the random model $\URM$ we show that
\emph{with high probability} ($\whp$, meaning with probability tending to one as $n\rightarrow
\infty$), for any constant $k\ge 1$, every random coloring has a rainbow matching.

\begin{theorem}
	\label{thm:whp}
Every random edge--coloring of $K_{n,n}$ in the $\URM$ with $s\ge n$ colors  has $\whp$ a rainbow matching.
\end{theorem}

To prove the Theorem~\ref{thm:whp} we use the second moment method on the random variable that
counts the number of rainbow matchings in the $\URM$ model. Observe that the same result can be
proven using the same idea for the $\RRM$ model.
\vspace{0.3cm}

 The paper is organized as follows. In Section~\ref{sec:asymp} we provide a proof for
Theorem~\ref{thm:given}. The random coloring models are defined in Section~\ref{sec:random}, where
we also prove Theorem~\ref{thm:random}. In the Subsection~\ref{ssc:whp} we display a prove for
Theorem~\ref{thm:whp}. Finally on Section~\ref{sec:open} we discuss about open problems about
rainbow matchings that arise from the paper.

\section{Asymptotic enumeration}\label{sec:asymp}

In this section we  prove Theorem~\ref{thm:given}. When $k=\omega (1)$ for $n\to \infty$, it
gives an asymptotically tight estimation of the probability that a random matching is rainbow. 
For constant $k$ the theorem provides exponential upper and lower bounds for this
probability.

\subsection{Lower bound}

One of the standard tools to give a lower bound for $\Pr( \cap_{i=1}^m \overline{A_j})$ is the
Local Lemma. In particular, as it is shown in~\cite{es1991}, it is convenient in our current
setting to use the Lopsided version of it.

Given a set of events $A_1,\dots,A_m$, a graph $H$ with vertex set $V(H) = \{ 1,\ldots ,m\}$ is
a \emph{lopsidependency graph}  for the events if, for each $i$ and each subset  $S\subseteq
\{j \mid ij\not\in E(H),\, j\neq i\}$, we have
$$
\Pr(A_i\mid \cap_{j\in S} \overline{A_j})\leq \Pr(A_i).
$$
Following Lu and Szekely  \cite{ls2009}, we adopt the more explanatory term {\it negative
dependency graph}  for this notion. We next recall the statement of the L\'ovasz Local Lemma we will
use. It includes an intermediate step, that appears in its proof, which will also be used later on.

\begin{lemma}[LLLL]\label{lem:LLLL} Let $\{A_1,\ldots ,A_m\}$ be events and let $H=(V,E)$ be a graph
on $\{1,\ldots ,m\}$ such that, for each $i$ and each $S\subseteq
\{j \mid ij\not\in E(H),\, j\neq i\}$,
$$
\Pr (A_i|\cap_{j\in S}\overline{A_j})\le P(A_i).
$$

Let $x_1,\ldots ,x_m\in (0,1)$. If, for each $i$,
\begin{equation}
\label{eq:hyp}
 \Pr (A_i)\le x_i\prod_{ij\in E(H)} (1-x_j),
\end{equation}
then, for each $T\subset [m]$ we have
\begin{equation}\label{eq:step}
\Pr (A_i|\cap_{j\in T} \overline{A_j})\le x_i.
\end{equation}
In particular, for each $S\subset [m]$ disjoint from $T$ we have
\begin{equation}\label{eq:step2}
\Pr (\cap_{i\in S}\overline{A_i}|\cap_{j\in T} \overline{A_j})\ge \prod_{i\in S}(1-x_i),
\end{equation}
and
\begin{equation}\label{eq:lll}
\Pr (\cap_{j\in [m]} \overline{A_j})\ge \prod_{j\in [m]}(1-x_j).
\end{equation}
\end{lemma}

Recall that ${\mathcal M}$ denotes  the family of  pairs of independent edges that have the
same color and, for each such pair $\{e,f\}\in {\mathcal M}$, we denote by $A_{e,f}$ the event that
the pair belongs to a perfect random matching $M$. We identify $\M$ with this set of events. We
consider the following dependency graph:

\begin{definition}\label{def:h} The rainbow dependency graph $H$ has  the family $\M$ as vertex
set. Two elements in $\M$ are adjacent in $H$ whenever the corresponding pairs of edges share some
end vertex in $K_{n,n}$.
\end{definition}

It is shown in Erd\H os and Spencer~\cite{es1991} that the graph $H$ defined above  is a
negative dependency graph. The following lower bound can be obtained in a similar way to Lu and
Szekely \cite[Lemma 2]{ls2009}. Recall that we consider edge--colorings of $K_{n,n}$ in which each
color appears at most $n/k$ times.

\begin{lemma}\label{lem:lower} With the above notations, if $k\ge 12$ then
$$
\Pr (\cap_{\{e,f\}\in\M}\overline{A_{e,f}})\ge e^{-(1+16/k)\mu},
$$
where $\mu=\sum_{\{e,f\}\in \M}\Pr(A_{e,f})$.

In particular, if $k=k(n)=\omega (\sqrt{n})$, then
$$
\Pr (\cap_{\{e,f\}\in\M}\overline{A_{e,f}})\ge (1+o(1))e^{-\mu}.
$$
\end{lemma}

\begin{proof}  Set ${\cal M}=\{A_1,\ldots ,A_m\}$. The size of $\M$ depends on the
configuration of the colors in $E(K_{n,n})$. In the worst case all the colors appear repeated  in
exactly $n/k$ disjoint edges. Thus,
\begin{equation*}
|\M|\leq kn \binom{n/k}{2}\sim\frac{n^3}{2k}
\end{equation*}

Since we are taking a random perfect matching, $p=\Pr(A_{i})=\tfrac{1}{n(n-1)}$ for each $i$. Then
\begin{equation}
\label{eq:mu}
\mu=\frac{|\M|}{n(n-1)}= \frac{1}{2}\left(1+\frac{1}{n-1}\right)\left(\frac{n}{k}-1\right) \le
\frac{n}{2k}
\end{equation}
Set $t=4/k$. Since $n\ge k\ge 12$ we have $t\le 1/3$ and $p\le 1/35$. It can be checked that, for $4/n<t<7/50$ and $0<p<1/35$, we have
 $$
 p e^{(1+4t)t}<1-e^{-(1+4t)p}.
$$
Choose $x_i$ in the interval $(p e^{(1+4t)t},1-e^{-(1+4t)p})$. For each $1\le i\le m$ we have
\begin{equation}\label{eq:xi}
\Pr (A_i)=p<x_ie^{-(1+4t)t}<x_i\prod_{ij\in E(H)}e^{-(1+4t)\Pr(A_j)}<x_i\prod_{ij\in E(H)}(1-x_j).
\end{equation}
Thus, by Lemma \ref{lem:LLLL},
$$
\Pr (\cap_{A_i\in {\M}} \overline{A_i})\ge \prod_{i=1}^m (1-x_i)\ge e^{-\left(1+16/k\right)\mu}.
$$
This proves the first part of the Lemma. In particular, since $\mu\le n/2k$,
$$
\Pr (\cap_{A_i\in {\M}} \overline{A_i})\ge  e^{-\mu}\left(1-\frac{16\mu}{k}\right)\ge e^{-\mu}\left(1-\frac{8n}{k^2}\right).
$$
so that, if $k=k(n)=\omega (\sqrt{n})$, then
$$
\Pr (\cap_{A_i\in {\M}} \overline{A_i})\ge  e^{-\mu}(1+o(1)).
$$

\end{proof}


\subsection{Upper bound}

Lu and Szekely~\cite{ls2009} propose a new enumeration tool using the Local Lemma.
Their objective is to find an upper bound for the non occurrence of rare events comparable with the
Janson inequality. In order to adapt the Local Lemma, they define a new type of parametrized
dependency graph: the \emph{$\varepsilon$-near dependency graph}.

Let $A_1,\dots,A_m$ a set of events. A graph $H$ with vertex set $\{A_1,\ldots ,A_m\}$ is an $\varepsilon$-near-positive dependency graph ($\varepsilon$-NDG) if,
\begin{enumerate}[i)]
 \item if $A_i\sim A_j$, then $\Pr(A_i\cap A_j)=0$.
 \item for any set $S\subseteq \{j:A_j\nsim A_i\}$ it holds $\Pr(A_i \mid \cap_{j\in S}
\overline{A_j})\geq (1-\varepsilon)\Pr(A_i)$.
\end{enumerate}

Condition $i)$ implies that only incompatible events can be connected. Condition $ii)$ says
that this set of non connected events  can not shrink the probability of $A_i$ too much.

\begin{theorem}[Lu and Szekely~\cite{ls2009}]
\label{thm:ub}
 Let $A_1 ,\dots, A_m$ be events with an $\varepsilon$--near--positive dependency graph $H$. Then we have,
\begin{equation*}
\Pr(\cap \overline{A_i})\leq \prod_{i} (1-(1-\varepsilon)\Pr(A_i)).
\end{equation*}
\end{theorem}

Observe that this upper bound  gives an exponential upper bound of the form
$e^{(1-\varepsilon)\mu}$.

Lu and Szekely~\cite{ls2009}  show also that an  $\varepsilon$--near--positive dependency
$H$ can be constructed using a family of matchings $\M$. Unfortunately the conditions of
\cite[Theorem 4]{ls2009} which would provide the upper bound in our case do not apply to our family
$\M$ of matchings. We give instead a direct proof for the upper bound which is inspired by their
approach. 

\begin{lemma}\label{lem:positive} The graph $H$ is an $\varepsilon$--near-positive dependency
graph with $\varepsilon=1- e^{-(2/k+32/k^2)}$.
\end{lemma}

\begin{proof} Set ${\cal M}=\{A_1,\ldots ,A_m\}$. The graph $H$ clearly satisfies condition i) in the definition of $\varepsilon$-NDG.
For condition ii) we want to show that, for each $i$ and each $T\subseteq
\{j \mid ij\not\in E(H),\, j\neq i\}$, we have the inequality
$$
\Pr (A_i|B)\ge (1-\epsilon)\Pr(A_i),
$$
where $B=\cap_{j\in T}\overline{A_j}$.
This is equivalent to show
$$
\Pr (B|A_i)\ge (1-\epsilon)\Pr(B).
$$
Let $\{a_1,\ldots ,a_n\}$ and $\{ b_1,\ldots ,b_n\}$ be the vertices of the two stable sets of
$K_{n,n}$. We may assume that $A_i$ consists of the two edges $a_{n-1}b_{n-1}, a_nb_n$. Then
$\{A_j:j\in T\}$ consists of a set of  $2$--matchings in $K_{n,n}-\{ a_{n-1},a_n,b_{n-1},b_n\}$.
This is the complete bipartite graph $K_{n',n'}$, $n'=n-2$, with an edge--coloring in which each
color appears at most $n'/k'$ times, where $k'= k(1+2/(n-2))$. Let us call $B'$ the event  $B$
viewed in $K_{n'n'}$ (dashes in notation indicate changing the probability space from random
matchings in $K_{n,n}$ to random matchings in $K_{n'n'}$), so that
\begin{equation}\label{eq:b'}
\Pr (B|A_i)=\Pr (B').
\end{equation}
Let $C_{r,s}$ denote the $2$--matching $a_{n-1}b_r, a_nb_s$, where $r\neq s$.  Define  $T_{r,s}\subset T$ in such a way that $\{A_j:j\in T_{r,s}\}$ are the $2$--matchings  in $\{A_j:j\in T\}$ which meet none of the two vertices $b_r, b_s$.
Set $B_{r,s}=\cap_{j\in T_{r,s}}\overline{{A_j}}$. Let us show that
\begin{equation}\label{eq:induction}
\Pr(B)=\frac{1}{n(n-1)}\sum_{r\neq s}\Pr (B_{r,s}'),
\end{equation}
where, as before, $B'_{r,s}$ denotes the event $B_{r,s}$ in the probability space of random matchings in $K_{n',n'}$.

We have
$$
\Pr(B)=\sum_{r\neq s}\Pr (B\cap C_{r,s})=\sum_{r\neq s}\Pr (B_{r,s}\cap C_{r,s}),
$$
Note that, since none of the matchings involved in $B$ meets vertices in $\{ a_{n-1},a_n,b_{n-1},b_n\}$, we have, for all $r,s$, $r\neq s$,
$$
\Pr (B_{r,s}|C_{r,s})=\Pr (B_{r,s}|C_{n-1,n}).
$$
Moreover, observe that $\Pr (B_{r,s}|C_{n-1,n})=\Pr (B'_{r,s})$. Therefore
$$
\Pr(B)=\sum_{r\neq s}\Pr (B_{r,s}| C_{r,s})\Pr (C_{r,s})=\frac{1}{n(n-1)}\sum_{r\neq s}\Pr (B_{r,s}|C_{n-1,n})= \frac{1}{n(n-1)}\sum_{r\neq s}\Pr (B_{r,s}'),
$$
giving equality \eqref{eq:induction}.

From inequality \eqref{eq:xi} we know that   $x'_j= 1- e^{-(1+16/k')p'}$ fulfills the hypothesis~\eqref{eq:hyp} of  the Local Lemma. We can now use the intermediate inequality \eqref{eq:step2} of the Lemma with $S=T\setminus T_{r,s}$ to obtain
\begin{equation}\label{eq:step2b}
\Pr (B_{r,s}')= \frac{\Pr (B')}{\Pr(\cap_{j\in S}\overline{A_j})}\le \Pr (B')\prod_{j\in S}(1-x_j')^{-1}.
\end{equation}
By combining \eqref{eq:b'} with \eqref{eq:induction} and \eqref{eq:step2b} we get

\begin{equation}\label{eq:epsilon}
\Pr(B|A_i)\ge \Pr (B)\prod_{j\in S}(1-x_j').
\end{equation}

Recall that $S=T\setminus T_{r,s}$ is  the set of  $2$-matchings in $\M'$ that are incident to
$b_r$ or $b_s$. The size of this set can be bounded independently from $r$ and $s$ by
$$
|S|\leq 2 n'\left(\frac{n'}{k'}-1\right) \leq 2\frac{n^2}{k}
$$
With our choice of    $x'_j= 1- e^{-(1+16/k')p'}\le 1-e^{-(1+16/k)p}$ (where   $p=1/n(n-1)$) we have
$$
\prod_{j\in S}(1-x_j')\ge  e^{-(1+16/k)p|S|}\ge e^{-(2/k+32/k^2)}.
$$
Therefore, by \eqref{eq:epsilon},
$$
 \varepsilon =   1- e^{-(2/k+32/k^2)},
$$
satisfies the conclusion of the Lemma.
\end{proof}

Now we are able to prove  Theorem~\ref{thm:given}.

\begin{proofof}[Proof of Theorem~\ref{thm:given}]
Set ${\cal M}=\{A_1,\ldots ,A_m\}$. By Lemma \ref{lem:positive}, the graph $H$ is an
$\varepsilon$--near--positive dependency graph with $\varepsilon =   1- e^{-(2/k+32/k^2)}$. It
follows from Theorem~\ref{thm:ub} that the probability of having a rainbow matching is upper bounded
by
$$
\Pr(\cap_{i\in [m]} \overline{A_i}) \leq \prod_{i\in [m]}
\left(1-(1-\varepsilon)\Pr (A_i)\right) \le e^{-(1-\epsilon)\mu}.
$$
By plugging in our value of $\varepsilon$ and by using $e^{-(2/k+32/k^2)}\ge 1-\frac{2}{k}-\frac{32}{k^2}$ we obtain
$$
\Pr(\cap_{i\in [m]} \overline{A_i}) \leq e^{-(1-2/k-12/k^2)\mu}.
$$
Combining this upper bound with the lower bound obtained in Lemma \ref{lem:lower} we obtain
$$
\exp \left\{-\left( 1+\frac{16}{k}\right) \mu\right\} \leq Pr(\cap \overline{A_i})\leq
\exp \left\{-\left( 1-\frac{2}{k}-\frac{12}{k^2}\right)\mu\right\}.
$$
This proves the first part of the Theorem.

In particular, since $\mu\le n/2k$, if $k=\omega (n^{1/2})$ we get
$$
\Pr(\cap_{i\in [m]} \overline{A_i}) \leq e^{-\mu}(1+o(1)),
$$
which matches the lower bound obtained in Lemma \ref{lem:lower}, thus proving the second part of the Theorem.
\end{proofof}

\section{Random colorings}
\label{sec:random}

In this section we will analyze the existence of rainbow matchings when the edge coloring of $K_{n,n}$ is given at random.

Recall that, in the uniform random model $\URM$, each edge of $K_{n,n}$ is given a color uniformly and
independently chosen from a set $C$ with $s$ colors, i.e. every possible coloring with at most $s$ colors appears with the same probability.

In the regular random model $\RRM$ a coloring is chosen
uniformly at random among all colorings of $E(K_{n,n})$ with equitable color classes of size
$n^2/s$. In order to construct a coloring in the $\RRM$ we use a complete bipartite graph
$H=(A,B)$, where $A$ contains $s$ blocks, each of size $n^2/s$, representing the colors and
$B$ is the set of edges of $K_{n,n}$. Every perfect matching in $H$ gives an equitable coloring of
$E(K_{n,n})$. Moreover, every equitable coloring of $E(K_{n,n})$  corresponds to the same
number of perfect matchings. Therefore, by selecting a random perfect matching in $H$ with the
uniform distribution, all equitable colorings have the same probability.

We established these two models since they simulate the worst situation in all the possible
colorings admitted in Theorem~\ref{thm:given}: the probability for a matching of being
rainbow only depends on the size of $|\M|$, and this set has its largest cardinality when there are
few colors with a maximum number of occurrences. This means that we have $s=nk$ colors with $n/k$
occurrences each.
Observe that in both models the expected size of each color class is also $n/k$, and in this sense,
they are contiguous to the hypothesis of Theorem~\ref{thm:given}. One can draw an analogy
between the $\URM$ and the Erd\H os-R\'enyi model $G(n,p)$ for random graphs, and also between the
$\RRM$ and the regular random graph $G(n,d)$.

\begin{proofof}[Proof of Theorem~\ref{thm:random}]

In the $\URM$ it is easy to compute the probability of having the rainbow property for a matching.
Let $M$ be a random perfect  matching, and $K_{n,n}$ provided with a random edge coloring with
$s\ge n$ colors. For the indicator variable $X_M$ that $M$ is rainbow we have:
\begin{eqnarray}
\Pr(X_M=1) & = & \frac{s}{s}\cdot\frac{s-1}{s}\cdot\frac{s-2}{s}\cdot\,\dots\,\cdot\frac{s-(n-1)}{s}
\nonumber\\
&=& \prod_{i=0}^{n-1} \left(1-\frac{i}{s}\right).\label{eq:URM}
\end{eqnarray}

For $s=n$ we can get directly from \eqref{eq:URM}
$$
\Pr(X_M=1)=\frac{n!}{n^n}\sim e^{-2\mu}.
$$
Assume $s>n$.
We have, for $0<x<1$,
\begin{equation}
\label{eq:approx}
	(1-x) = \exp\left\{ \log{(1-x)} \right\} 
\end{equation}
Therefore,
\begin{eqnarray*}
\Pr(X_M=1) & = & \prod_{i=1}^{n-1} \exp\left\{\log{\left(1-\frac{i}{s}\right)}\right\} \\
& = & \exp\left\{ \sum_{i=1}^{n-1} \log{\left(1-\frac{i}{s}\right)}  \right\} \\
& \sim & \exp\left\{ \int_0^n \log{\left(1-\frac{x}{s}\right)} dx \right\}.
\end{eqnarray*}
We use
$$
\int_0^t \log\left(1-x\right)dx=(t-1)\log\left(1-t\right)-t
$$
By writing $k=s/n$
\begin{eqnarray*}
\Pr(X_M=1) & = & \exp\left\{ \left((k-1)\log{\left(\frac{k}{k-1}\right)}-1\right)n\right\}\\
 & \sim & \exp\left\{ - 2k\left( 1 - (k-1)\log \left(\frac{k}{k-1}\right)\right) \mu\right\}.
\end{eqnarray*}
since $\mu\sim\tfrac{n}{2k}$.

It must be stressed that this result is consistent with the ones in Theorem \ref{thm:given}. When $k=1$ we
have $\Pr(X_M=1)= e^{-2\mu}$, while $\lim_{k\rightarrow \infty} \Pr(X_M=1)= e^{-\mu}$.
Observe that, in this case, $\E (|\M|)$ is not exactly the same as for a given coloring. This is due
to the variance on the number of occurrences of each color, but does not have a significant
importance.

To study the property that a random selected matching is rainbow in the $\RRM$ we express the
equitable edge colorings through permutations $\sigma\in Sym (n^2)$.
Then, probability for a matching $M$ of being rainbow is,

\begin{eqnarray*}
\Pr(X_M=1) & = &
\frac{n^2}{n^2}\cdot\frac{n^2-\tfrac{n^2}{s}}{n^2-1}\cdot\frac{n^2-2\tfrac{n^2}{s}}{n^2-2}
\cdot\,\dots\,\cdot\frac {
n^2-(n-1)\tfrac{n^2}{s}}{n^2-(n-1)} \\
&=&\prod_{i=0}^{n-1} \left(1-\frac{i(n^2-s)}{s(n^2-i)}\right)\\
&=&  \exp \left\{\sum_{i=0}^{n-1} \log\left(1-\frac{i(n^2-s)}{s(n^2-i)}\right) \right\} \quad
\mbox{(by~\eqref{eq:approx})}\\
&\sim&  \exp \left\{\int_{0}^{n} \log\left(1-\frac{x(n^2-s)}{s(n^2-x)}\right) dx \right\}.
\end{eqnarray*}
If $s=n$ we have
$$
\int_{0}^{n} \log\left(1-\frac{x(n-1)}{(n^2-x)}\right) dx = -n(n-1)\log{\left(\frac{n}{n-1}\right)},
$$
which, by using the Taylor expansion of the logarithm, gives
$$
\Pr(X_M=1)  = e^{-(1+o(1))2\mu}
$$
In the case where $s>n$, and using $k=s/n$, we have
\begin{eqnarray*}
\int_{0}^{n} \log\left(1-\frac{x(n^2-s)}{s(n^2-x)}\right) dx &= &
\left((k-1)\log{\left(\frac{k}{k-1}\right)}-(n-k)\log{\left(\frac{n}{n-1}\right)}\right) n\\
&= & \left((k-1)\log{\left(\frac{k}{k-1}\right)} -1 +o(1)\right)n.
\end{eqnarray*}
Hence
\begin{eqnarray*}
\Pr(X_M=1) & = & \exp\left\{ - 2k\left(1- (k-1)\log \left(\frac{k}{k-1}\right) + o(1)\right)
\mu\right\}.
\end{eqnarray*}

\end{proofof}
Note that, for both models of random edge colorings, the probability that a fixed perfect
matching is rainbow the same (up to a $o(1)$ term).
Thus, in spite of being different models, the probability of having a rainbow matching is
similar.

In general it is not true that $\Pr(X_M=1)=e^{-(1+o(1))\mu}$ but, if $k=\omega(1)$, then
$\Pr(X_M=1)\rightarrow e^{-\mu}$ for both models since
$$
2k\left(1- (k-1)\log \left(\frac{k}{k-1}\right) \right) =
1+O\left(\frac{1}{k}\right)
$$
This is natural since, when $k$ is large, the number of bad events decreases and the model  behaves
like in the case they were independent.

Observe that for the two random models we obtain the exact asymptotic  value of the probability, while bounds provided by Theorem~\ref{thm:given} (when the size $|\M|$ of the set of bad events is maximum) are not sharp, although consistent with the values for the random models.  Since the result proven for fixed colorings does only depend on the size of $\M$, the probability for the random model $\RRM$ should be exactly the same.

\section{Existence of rainbow matchings}
\label{ssc:whp}

The aim of this Section is to prove that $\whp$ there exists a rainbow matching for a given
random coloring of $E(K_{n,n})$ with $s\ge n$ colors. We only consider the $\URM$, but the results can be adapted to the $\RRM$. The number of rainbow matchings is counted by $X=\sum X_M$, which, according to Theorem \ref{thm:random}, has expected value
\begin{equation}
\label{eq:expected}
\mathbb{E}(X)= n! \Pr(X_M=1)\sim n! \exp\left\{ - 2k\left( 1 - (k-1)\log \left(\frac{k}{k-1}\right)\right) \mu\right\}.
\end{equation}
In order to have a rainbow matching we just need that $X\neq 0$. 

Given two perfect matchings $M$ and $N$, the events that they are rainbow are positively correlated,
\begin{equation}
\Pr(X_M=1\mid X_N=1)\geq \Pr(X_M=1).
\end{equation}

\begin{proofof}[Proof of Theorem~\ref{thm:whp}]

To show that there exists some rainbow matching $\whp$ we will use the second moment method.
Let $X$ be a random variable with expected value $\mu$ and variance $\sigma^2$. Then, the Chebyshev
inequality asserts that
\begin{equation}\label{eq:cheby}
 \Pr(|X-\mu|> \alpha \sigma) \leq \frac{1}{\alpha^2}
\end{equation}

In particular if $\alpha=\tfrac{\mu}{\sigma}$,
\begin{equation}\label{eq:cheby2}
 \Pr(X=0)\leq \Pr(|X-\mu|>\mu) \leq \frac{\sigma^2}{\mu^2}
\end{equation}

Observe that $X=0$ is equivalent to the non existence of any rainbow matching. Therefore, we need to
compute $\sigma^2(X)$ and show that it is asymptotically smaller than $\E(X)^2$. Note that
$$
\mathbb{E}(X^2) = \sum_{M,N} \mathbb{E}(X_MX_N)
$$

Let $M$ and $N$ two fixed matchings, then
$$
\mathbb{E}(X_MX_N)=\Pr(X_M)\Pr(X_{N\setminus M})
$$

Given a fixed matching $M$ and a fixed intersection size $t$, we claim there are at most
$e^{-1}\binom{n}{t}(n-t)!$ matchings $N$, such that $|M\cap N|=t$. There are $\binom{n}{t}$ ways of
choosing which edges will be shared and once this edges have been fixed, at most $e^{-1}(n-t)!$ ways
of completing the matching. Suppose that $\tau=\sigma_N\mid_{N\setminus M} \in \S_{n-t}$ is the
permutation for extending the matching in the disjoint part.
Since $N$ has intersection exactly $t$ with $M$, not any permutation is valid. We can
assume $\wl$ that $M$ is given by $\sigma_M=Id$, and therefore $\tau$ must be a derangement.
Classical results state that the proportion of derangements in permutations of any length is at
most $e^{-1}$. This concludes our claim.

Hence,
$$
\E(X^2)= e^{-1}n! \sum_{t=0}^{n} \binom{n}{t}(n-t)! \Pr(X_M)\Pr(X_{N\setminus M})
$$
%

%

Since $\sigma^2=\E(X^2)-\E(X)^2$,
\begin{eqnarray*}
\frac{\sigma^2(X)}{\E(X)^2} &=& \frac{e^{-1}n! \sum_{t=0}^{n} \binom{n}{t}(n-t)!
\Pr(X_M)\Pr(X_{N\setminus M})}{(n! \Pr(X_M))^2} -1\\
& = & e^{-1}\sum_{t=0}^{n}\frac{1}{t!}\frac{\Pr(X_{N\setminus M})}{\Pr(X_M)}-1\\
\end{eqnarray*}

Given $X_M$, we know that the edges of $M\cap N$ are rainbow. In the remaining $n-t$ edges to
color, we must avoid the $t$ colors that appear in $M\cap N$
$$
\Pr(X_N\mid X_M) = \prod_{i=t}^{n-1} \left(1-\frac{i}{s}\right)
$$

Then,
$$
f(s) = \sum_{t=0}^{n}\frac{1}{t!}\frac{\Pr(X_{N\setminus M})}{\Pr(X_M)} \sim
\sum_{t=0}^{n}\frac{1}{t!}e^{(1+O(1/k))\tfrac{t^2}{2s}} \sim
\sum_{t=0}^{\infty}\frac{1}{t!}e^{(1+O(1/k))\tfrac{t^2}
{2s}}$$

If the number of colors $s=\omega(1)$, then $f(s)\rightarrow e$. Observe that $s\geq n$. Otherwise,
$\Pr(X_M)=0$ in the Equation~\eqref{eq:URM}.

Hence $\tfrac{\sigma^2}{\mu^2}\rightarrow 0$ and the theorem holds.
\end{proofof}

Actually,
$$
f(s)=\frac{1}{s}+O(s^{-2})
$$

and Equation~\eqref{eq:cheby2} also provides an upper bound estimation for the probability $p$ that
a random coloring has no rainbow matchings of the type
$$
p\leq (1+o(1))\frac{1}{n} 
$$
Observe that the proportion of Latin squares among the set of square matrices with $n$ symbols is
of the order of $e^{-n^2}$ (see e.g.~\cite{vw2001}), so that this estimation falls short to prove
an asymptotic version of the original conjecture of Ryser.

\section{Open Problems}
\label{sec:open}

On the ennumeration of Rainbow matchings, it would be interesting to prove exact upper and lower
bounds for the case where the number of occurences of each color is at most $k$, with constant $k$.
Theorem~\ref{thm:given} provides exponential upper and lower bounds as long as $k\geq 12$, but both
are asymptotically equal if and only if $k=\omega(1)$. 

A related problem is to improve the lower bound $k\geq 4e$ given by Erd\H os
and Spencer~\cite[Theorem 2]{es1991} for the existence of rainbow matchings.

On the other hand, another really interesting problem is to prove that almost all latin squares
have a latin transversal, i.e. the asymptotic version of the Ryser conjecture. We have stablished a
probabilistic way to approach the problem.  Unfortunately, as far as we know, there
are no random models for latin squares. Some results on generating random latin squares can be found
in~\cite{mw1991,jm1996}. Nevertheless, there are some almost sure results on Latin
squares~(see e.g.\cite{mw1999,cgw2008}).

\bibliography{rainbow}
\bibliographystyle{amsplain}

\end{document}